\documentclass[reqno,draft]{amsart}
\usepackage[american]{babel}
\usepackage{amssymb,euscript}
\numberwithin{equation}{section}

\newtheorem{theorem}{Theorem}[section]
\newtheorem{lemma}[theorem]{Lemma}
\newtheorem{corollary}[theorem]{Corollary}
\theoremstyle{definition}
\newtheorem{definition}[theorem]{Definition}
\theoremstyle{remark}
\newtheorem{remark}[theorem]{Remark}

\DeclareMathOperator{\sgn}{sign}

\begin{document}

\title[Zakharov--Kuznetsov equation]
{An initial-boundary value problem in a strip for two-dimensional Zakharov--Kuznetsov--Burgers equation}

\author[A.V.~Faminskii]{Andrei~V.~Faminskii}

\thanks{The work was supported by Project 333, State Assignment in the field of scientific activity implementation of Russia}

\address{Department of Mathematics, Peoples' Friendship University of Russia,
Miklukho--Maklai str. 6, Moscow, 117198, Russia}

\email{afaminskii@sci.pfu.edu.ru}



\date{}

\begin{abstract}
An initial-boundary value problem in a strip with homogeneous Dirichlet boundary conditions for two-dimensional  Zakharov--Kuznetsov--Burgers equation is considered. Results on global well--posedness and long-time decay of solutions  in $H^s$ for $s\in [0,2]$ are established.
\end{abstract}

\maketitle

\section{Introduction. Description of main results}\label{S1}
 
The goal of this paper is to study global well-posedness and large-time decay of solutions for an initial-boundary value problem on a strip $\Sigma=\mathbb R\times (0,L)=\{(x,y): x\in \mathbb R, 0<y<L\}$ of a given width $L$ for an equation
\begin{equation}\label{1.1}
u_t+u_{xxx}+u_{xyy}+uu_x-\delta(u_{xx}+u_{yy})=0,\qquad \delta=\text{const}>0,
\end{equation}
with an initial condition 
\begin{equation}\label{1.2}
u(0,x,y)=u_0(x,y),\qquad (x,y)\in\Sigma,
\end{equation}
and homogeneous Dirichlet boundary conditions 
\begin{equation}\label{1.3}
u(t,x,0)=u(t,x,L)=0,\qquad (t,x)\in \mathbb R_+\times\mathbb R.
\end{equation}

This equation is referred as Zakharov--Kuznetsov--Burgers equation because it originates from Zakharov--Kuznetsov equation
\begin{equation*}
u_t+u_{xxx}+u_{xyy}+uu_x=0
\end{equation*}
supplemented with parabolic terms as in Burgers equation. Zakharov--Kuznetsov equation is a multi-dimensional generalization of Korteweg--de~Vries equation
\begin{equation*}
u_t+u_{xxx}+uu_x=0
\end{equation*} 
and is considered as a model equation for non-linear waves propagating in dispersive media in the preassigned direction $x$ with deformations in the transverse direction $y$. For the first time it was derived in \cite{ZK} for ion-acoustic waves in magnetized plasma. Equation \eqref{1.1} can be treated as as a model equation for non-linear wave processes including both dispersion and dissipation. 

The theory of well-posedness for Zakharov--Kuznetsov equation is most developed for the initial value problem and for initial-boundary value problems on domains of a type $I\times \mathbb R$, where $I$ is an interval (bounded or unbounded) on the variable $x$, that is the variable $y$ varies in the whole line (see, for example, references in \cite{F14, STW}).  

However, it seems more natural from the physical point of view to consider domains, where the variable $y$ varies in a bounded interval. In fact, it turned out that it is more difficult to study such problems than the aforementioned ones and there are only a few results on the matter.

In \cite{LPS} an initial-boundary value problem for Zakharov--Kuznetov equation in the strip $\Sigma$ with periodic boundary conditions is considered and local well-posedness result is established in the spaces $H^s$ for $s>3/2$. Initial-boundary value problems in the strip $\Sigma$ with homogeneous boundary conditions of different types: Dirichlet, Neumann or periodic are studied in \cite{BF13} and results on global existence and uniqueness in classes of weak solutions with power weights at $+\infty$ are established. For Dirichlet boundary conditions these results are supplemented in \cite{F14} with results of exponential long-time decay of small solutions in $L_2$ spaces with exponential weights at $+\infty$. In \cite{LT, L} an initial-boundary value problem in a half-strip $\mathbb R_+\times (0,L)$ with homogeneous Dirichlet boundary conditions is studied and global well-posedness in spaces $L_2$ and $H^1$ with exponential weights when $x\to +\infty$ as well as exponential decay as $t\to +\infty$ of small solutions are proved. For a bounded rectangle global well-posedness results can be found in \cite{DL, STW} and exponential long-time decay of small solutions in $L_2$ --- in \cite{DL}.

Of course, the presence of any regularization can improve results on well-posedness and long-time decay of solutions. 

In \cite{F14} the problem in the strip $\Sigma$ with initial and boundary conditions \eqref{1.2}, \eqref{1.3} is studied for an equation
$$
u_t+u_{xxx}+u_{xyy}+uu_x-(a_1(x,y)u_x)_x -(a_2(x,y)u_y)_y +a_0(x.y)u=0.
$$
The functions $a_1, a_2$ are assumed to be non-negative, that is the parabolic damping can degenerate. Certain results on global existence and uniqueness of weak solutions in $L_2$ and $H^1$ spaces (possibly weighted at $+\infty$) as well as their long-time decay in $L_2$ (not only for small solutions) are established when the parabolic damping is effective either at both inifinities or only at $+\infty$ or only at$-\infty$ or even be absent. For example, for $u_0\in L_2(\Sigma)$ and $a_j\in L_\infty(\Sigma)$ satisfying inequalities
\begin{gather*}
a_2(x,y)\geq \beta_2(x)\geq 0,\quad a_0(x,y)\geq \beta_0(x) \quad\forall (x,y)\in\Sigma,\\
\frac{\pi^2\beta_2(x)}{L^2}+\beta_0(x)\geq \beta=\text{const}>0\qquad\forall x\in\mathbb R
\end{gather*}
(that is either dissipation or absorption must be effective at every point) 
there exists a global solution such that
$$
\|u(t,\cdot,\cdot)\|_{L_2(\Sigma)} \leq e^{-\beta t}\|u_0\|_{L_2(\Sigma)}\qquad \forall t\geq 0.
$$
If dissipation is effective at both infinities, that is $a_1, a_2\geq a=\text{const}>0$ for $|x|\geq R$, the problem is globally well-posed and similar exponential decal in $L_2(\Sigma)$ is valid without any absorption ($a_0\geq 0$).

In  \cite{L1} the same initial-boundary value problem is studied for an equation
$$
u_t+u_{xxx}+u_{xyy}+uu_x-u_{xx}=0
$$
and results on global well-posedness in certain class of weak and regular solutions (decaying exponentially as $x\to+\infty$) as well as their exponential long-time decay in $L_2$ (for weak solutions) and $H^1$ (for regular ones) norms are obtained.

Note that without any additional damping of Zakharov--Kuznetsov equation the  long-time decay of solutions to the considered problem is impossible even in $L_2$ because of the conservation law
\begin{equation*}
\iint_\Sigma u^2(t,x,y)\,dxdy=\text{const}.
\end{equation*}

In the present paper results on global well-posedness and long-time decay of solutions to problem \eqref{1.1}--\eqref{1.3} are established in the spaces $H^s(\Sigma)$ for $s\in [0,2]$.

Introduce the following notation. For an integer $k\geq 0$ let
\begin{equation*}
|D^k\varphi|=\Bigl(\sum_{k_1+k_2=k}(\partial^{k_1}_x\partial_y^{k_2}\varphi)^2\Bigr)^{1/2}, \qquad
|D\varphi|=|D^1\varphi|.
\end{equation*}

Let $L_p=L_p(\Sigma)$ for $p\in [1,+\infty]$, $H^s=H^s(\Sigma)$, $H^s_0=H^s_0(\Sigma)$ for $s\in\mathbb R$.

For any $T_1<T_2$ let $\Pi_{T_1,T_2} = (T_1,T_2)\times \Sigma$, let $\Pi_T=\Pi_{0,T}$, $\Pi=\Pi_{+\infty} = \mathbb R_+ \times \Sigma$.
For $s\geq 0$ define a functional space 
$$
X^s(\Pi_{T_1,T_2}) =C([T_1,T_2];H^s)\cap L_2(T_1,T_2;H^{s+1}\cap H^1_0).
$$
We construct solutions to the considered problem lying in spaces $X^s(\Pi_T)$ for any $T>0$ if $s\in [0,2]$.

The main result of the paper is the following theorem.

\begin{theorem}\label{T1.1}
Let $u_0\in H^s$ for a certain $s\in [0,2]$ and, in addition, $u_0(x,0)=u_0(x,L)\equiv 0$ if $s>1/2$ and 
$\bigl(y^{-1/2}+(L-y)^{-1/2}\bigr)u_0\in L_2$ if $s=1/2$. Then there exists a unique solution to problem \eqref{1.1}--\eqref{1.3} $u \in X^s(\Pi_T)$ for any $T>0$. Moreover, there exist a constant $\beta(s)>0$ and a function $\sigma_s(\theta)$ nondecreasing with respect to $\theta\geq 0$ such that
\begin{equation}\label{1.4}
\|u(t,\cdot,\cdot)\|_{H^s} \leq \sigma_s(\|u_0\|_{H^{s_0}})e^{-\beta(s)t}\|u_0\|_{H^s} \qquad \forall\ t\geq 0,
\end{equation}
where $s_0=0$ if $s\in [0,1]$ , $s_0=1$ if $s\in (1,2]$.
\end{theorem}

Compare this result with the one-dimensional case from \cite{CCKR} (in fact, the present paper is inspired by that one),
where the initial value problem is considered for damped Korteweg--de~Vries--Burgers equation
$$
u_t+u_{xxx}+uu_x-u_{xx}+a(x)u=0.
$$
The function $a\equiv a_0+a_1$, where $a_0=\text{const}>0$, $a_1\in H^1(\mathbb R)\cap L_p(\mathbb R)$ for $p\in [1,+\infty)$ and $\|a_1\|_{L_p(\mathbb R)}$ is small in some sense (depending on $p$ and $a_0$). Then for the initial data from $H^s(\mathbb R)$, $s\in [0,3]$, the considered problem is globally well-posed with exponential decay as $t\to +\infty$ of solutions also in the space $H^s(\mathbb R)$. Note that the function $a$ is allowed to change sign but, of course, the presence of certain absorption is provided by the constant $a_0$.

Moreover, for pure Korteweg--de~Vries--Burgers equation ($a\equiv 0$) exponential decay of solutions to the initial value problem is in general case impossible even in $L_2(\mathbb R)$, because it is proved in \cite{ABS} that for $u_0\in L_2(\mathbb R)\cap L_1(\mathbb R)$ a corresponding solution to the initial value problem satisfies an inequality
\begin{equation*}
\|u(t,\cdot)\|_{L_2(\mathbb R)}\leq c(1+t)^{-1/4}\qquad \forall t\geq 0
\end{equation*}
and this result is sharp, so here dissipation without absorption ensures only power decay.

The idea that homogeneous Dirichlet boundary conditions in the horizontal strip of a finite width provide internal dissipation for Zakharov--Kuznetsov  equation which, in particular, yields exponential long-time decay of solutions was for the first time used in \cite{LT}.

Note that it is shown in \cite{CCFN} that the exponential long-time decay of solutions in $L_2(\mathbb R)$ holds for the initial value problem for damped Korteweg--de~Vries equation
$$
u_t+u_{xxx}+uu_x+a(x)=0
$$
even in the case of a localized absorption, that is if $a(x)\geq 0$ $\forall x\in\mathbb R$, $a(x)\geq a_0=\text{const}>0$ for $|x|\geq R$.

Long-time behavior of solutions at $H^2(\mathbb R)$ level for the initial value problem for generalized Korteweg--de~Vries--Burgers equation with constant absorption with the use of the global attractors theory is studied in \cite{dlotko}.

Further let $\eta(x)$ denotes a cut-off function, namely, $\eta$ is an infinitely smooth non-decreasing on $\mathbb R$ function such that $\eta(x)=0$ when $x\leq 0$, $\eta(x)=1$ when $x\geq 1$, $\eta(x)+\eta(1-x)\equiv 1$.

We omit limits of integration in integrals over the whole strip $\Sigma$.

The following interpolating inequality from \cite{LSU} is crucial for the study.

\begin{lemma}\label{L1.2}
Let $k$ be natural, $m\in [0,k)$ -- integer, $q\in [2,+\infty]$ if $k-m\geq 2$ and $q\in [2,+\infty)$ in other cases. Then there exists a constant $c>0$ such that for every function $\varphi(x,y)\in H^k$ the following inequality holds
\begin{equation}\label{1.5}
\bigl\| |D^m\varphi|\bigr\|_{L_q} \leq c 
\bigl\| |D^k\varphi|\bigr\|^{2s}_{L_2}
\bigl\| \varphi\bigr\|^{1-2s}_{L_2} 
+c\bigl\|\varphi\bigr\|_{L_2},
\end{equation}
where $\displaystyle{s=\frac{m+1}{2k}-\frac{1}{kq}}$. 
\end{lemma}

The use of nonlinear interpolation in this paper is based on the following
result from \cite{Tar}.

\begin{lemma}[Tartar]\label{L1.3}
Let $B_0^j$ and $B_1^j$ be Banach spaces such, that 
$B_1^j\subset B_0^j$ with continuous inclusion mappings, $j=1, 2$. Let
$B_\theta^j =(B_0^j,B_1^j)_{\theta,2}$, $\theta\in (0,1)$, be a space,
constructed by the method of real interpolation. Assume, that an operator
$\mathcal A$ maps $B_0^1$ into $B_0^2$, $B_1^1$ into $B_1^2$, where for any
$f,g \in B_0^1$
$$
\|\mathcal A f - \mathcal A g\|_{B_0^2} \leq 
c_1(\|f\|_{B_0^1},\|g\|_{B_0^1})\|f-g\|_{B_0^1},
$$
and for any $h\in B_1^1$
$$
\|\mathcal A h\|_{B_1^2} \leq c_2(\|h\|_{B_0^1})\|h\|_{B_1^1}.
$$
Then for any $\theta\in (0,1)$ the operator $\mathcal A$ maps $B_\theta^1$ into
$B_\theta^2$ and for any $f\in B_\theta^1$
$$
\|\mathcal A f\|_{B_\theta^2} \leq c(\|f\|_{B_0^1})\|f\|_{B_\theta^1},
$$
where all the functions $c_1, c_2, c$ are nondecreasing with respect to their arguments.

\end{lemma}

For the decay results we need Steklov inequality in such a form: for $\psi\in H_0^1(0,L)$
\begin{equation}\label{1.6}
\int_0^L \psi^2(y)\,dy \leq \frac{L^2}{\pi^2} \int_0^L \bigl(\psi'(y)\bigr)^2\,dy.
\end{equation}

The paper is organized as follows. Auxiliary linear problems are considered in Section~\ref{S2}. Section~\ref{S3} is devoted to global well-posedness of the original problem. Decay of solutions is studied in Section~\ref{S4}.

\section{An auxiliary linear problem}\label{S2}

In an arbitrary layer $\Pi_T$ consider a linear equation 
\begin{equation}\label{2.1}
u_t+u_{xxx}+u_{xyy}-\delta(u_{xx}+u_{yy})=f(t,x,y)
\end{equation}
and set initial and boundary conditions \eqref{1.2}, \eqref{1.3}.

Introduce certain additional function spaces. Let $\EuScript S(\overline{\Sigma})$ be a space of infinitely smooth in $\overline{\Sigma}=\mathbb R\times [0,L]$ functions $\varphi(x,y)$ such that $\displaystyle{(1+|x|)^n|\partial^k_x\partial^l_y\varphi(x,y)|\leq c(n,k,l)}$ for any integer non-negative $n, k, l$ and all $(x,y)\in \overline{\Sigma}$.

\begin{lemma}\label{L2.1}
Let $u_0\in \EuScript S(\overline{\Sigma})$, $f\in C^\infty\bigl([0,T]; \EuScript S(\overline{\Sigma})\bigr)$ and for any integer $j\geq 0$
\begin{equation*}
\partial^{2j}_y u_0\big|_{y=0}=\partial^{2j}_y u_0\big|_{y=L}=0, \qquad 
\partial^{2j}_y f\big|_{y=0}=\partial^{2j}_y f\big|_{y=L}=0.
\end{equation*}
Then there exists a unique solution to problem \eqref{2.1}, \eqref{1.2}, \eqref{1.3} $u\in C^\infty\bigl([0,T]; \EuScript S(\overline{\Sigma})\bigr)$.
\end{lemma}

\begin{proof}
For any natural $l$ let $\psi_l(y)\equiv\sqrt{\frac{2}{L}}\sin{\frac{\pi l}{L} y}$, $\lambda_l=\left(\frac{\pi l}L\right)^2$.   
Then a solution to the considered problem can be written as follows:
\begin{equation*}
u(t,x,y)=\frac{1}{2\pi}\int_{\mathbb R}\, \sum_{l=1}^{+\infty}e^{i\xi x}\psi_l(y)\widehat{u}(t,\xi,l)\,d\xi,
\end{equation*}
where
\begin{equation*}
\widehat{u}(t,\xi,l)\equiv \widehat{u_0}(\xi,l)e^{\left(i(\xi^3+\xi\lambda_l)
-\delta(\xi^2+\lambda_l)\right)t} +
\int_0^t\widehat{f}(\tau,\xi,l)e^{\left(i(\xi^3+\xi\lambda_l)-\delta(\xi^2+\lambda_l)\right)(t-\tau)}\,d\tau,
\end{equation*}
\begin{equation*}
\widehat{u_0}(\xi,l)\equiv\iint e^{-i\xi x}\psi_l(y)u_0(x,y)\,dxdy,\quad
\widehat f(t,\xi,l)\equiv\iint e^{-i\xi x}\psi_l(y) f(t,x,y)\,dxdy,
\end{equation*}
and, obviously, $u\in C^\infty([0,T],\EuScript S(\overline{\Sigma}))$.
\end{proof}

Next, consider generalized solutions. Let $u_0\in \EuScript S'(\overline{\Sigma})$, $f \in \bigl(C^\infty([0,T]; \EuScript S(\overline{\Sigma}))\bigr)'$.

\begin{definition}\label{D2.2}
A function $u \in \bigl(C^\infty([0,T]; \EuScript S(\overline{\Sigma}))\bigr)'$ is called a generalized solution to problem \eqref{2.1}, \eqref{1.2}, \eqref{1.3}, if for any function $\phi \in C^\infty\bigl([0,T]; \EuScript S(\overline{\Sigma})\bigr)$, such that $\phi|_{t=T}=0$ and $\phi|_{y=0}=\phi|_{y=L}=0$, the following equality holds:
\begin{equation}\label{2.3}
\langle u,\phi_t+\phi_{xxx}+\phi_{xyy}+\delta\phi_{xx}+\delta\phi_{yy}\rangle +
\langle f,\phi\rangle +\langle u_0,\phi|_{t=0}\rangle=0.
\end{equation}
\end{definition}

\begin{lemma}\label{L2.3}
A generalized solution to problem \eqref{2.1}, \eqref{1.2}, \eqref{1.3} is unique.
\end{lemma}

\begin{proof} 
The proof is implemented by standard H\"olmgren's argument on the basis of Lemma~\ref{L2.1}.
\end{proof}

\begin{lemma}\label{L2.4}
Let $u_0\in L_2$, $f\equiv f_0+f_{1x}+f_{2y}$, where $f_0\in L_1(0,T; L_2)$, $f_1,f_2\in L_2(\Pi_T)$. Then there exists a (unique) generalized solution to problem \eqref{2.1}, \eqref{1.2}, \eqref{1.3} $u\in X^0(\Pi_T)$. Moreover, for any $t\in (0,T]$ 
\begin{equation}\label{2.3}
\|u\|_{X^0(\Pi_t)}
\leq c(T,\delta) \left[\|u_0\|_{L_2}+\|f_0\|_{L_1(0,t;L_2)}
+\|f_1\|_{L_2(\Pi_t)}
+\|f_2\|_{L_2(\Pi_t)}\right]
\end{equation}
and
\begin{multline}\label{2.4}
\iint u^2(t,x,y)\,dxdy +2\delta \int_0^t\!\! \iint (u_x^2+u_y^2)\,dxdyd\tau = 
\iint u_0^2\,dxdy \\+
2\int_0^t\!\! \iint (f_0u-f_1u_x-f_2u_y)\,dxdyd\tau.
\end{multline}
\end{lemma}

\begin{proof}
It is sufficient to consider smooth solutions from Lemma~\ref{L2.1} because of linearity of the problem.

Multiplying \eqref{2.1} by $2u(t,x,y)$ and integrating over $\Sigma$ we obtain an equality
\begin{equation}\label{2.5}
\frac{d}{dt}\iint u^2 \,dxdy
+2\delta\iint(u_x^2+u_y^2) \,dxdy 
=2\iint (f_0 u-f_1u_x-f_2u_y)\,dxdy, 
\end{equation}
whence \eqref{2.3} and \eqref{2.4} are immediate.
\end{proof}

\begin{lemma}\label{L2.5}
Let $u_0\in H^1_0$, $f\equiv f_0+f_1$, where $f_0\in L_1(0,T;H^1_0)$, $f_1\in L_2(\Pi_T)$. Then there exists a (unique) generalized solution to problem \eqref{2.1}, \eqref{1.2}, \eqref{1.3} $u\in X^1(\Pi_T)$. Moreover, for any $t\in (0,T]$
\begin{equation}\label{2.6}
\|u\|_{X^1(\Pi_t)}
\leq c(T,\delta) \left[\|u_0\|_{H^1}+\|f_0\|_{L_1(0,t_0;H^1)}
+\|f_1\|_{L_2(\Pi_t)}\right]
\end{equation}
and
\begin{multline}\label{2.7}
\iint (u_x^2+u_y^2)\,dxdy +2\delta\int_0^t\!\! \iint (u_{xx}^2+2u^2_{xy}+u^2_{yy})\,dxdyd\tau =
\iint (u_{0x}^2+u_{0y}^2)\,dxdy  \\+
2\int_0^t\!\! \iint (f_{0x}u_x+f_{0y}u_y-f_1u_{xx}-f_1u_{yy})\,dxdyd\tau.
\end{multline}
\end{lemma}

\begin{proof}
In the smooth case multiplying \eqref{2.1} by $-2\bigl(u_{xx}(t,x,y)+u_{yy}(t,x,y)\bigr)$ and integrating over $\Sigma$ one obtains an equality
\begin{multline}\label{2.8}
\frac{d}{dt}\iint(u_x^2+u_y^2) \,dxdy
+2\delta\iint(u^2_{xx}+2u^2_{xy}+u^2_{yy}) \,dxdy \\
=2\iint(f_{0\,x}u_x+f_{0\,y}u_y) \,dxdy 
-2\iint f_1(u_{xx}+u_{yy})\,dxdy,
\end{multline}
whence \eqref{2.6} and \eqref{2.7} follows.
\end{proof}

\begin{lemma}\label{L2.6}
Let the hypothesis of Lemma~\ref{L2.5} be satisfied. Then for the solution to problem \eqref{2.1}, \eqref{1.2}, \eqref{1.3} $u\in X^1(\Pi_T)$ for any $t\in (0,T]$
\begin{multline}\label{2.9}
-\frac 13 \iint u^3(t,x,y)\,dxdy
+2\int_0^t\!\! \iint uu_x(u_{xx}+u_{yy})\,dxdyd\tau \\
+\delta \int_0^t\!\! \iint u^2(u_{xx}+u_{yy})\,dxdyd\tau =
-\frac 13 \iint u_0^3\,dxdy - \int_0^t\!\! \iint fu^2\,dxdyd\tau.
\end{multline}
\end{lemma}

\begin{proof}
In the smooth case multiplying \eqref{2.1} by $-u^2(t,x,y)$ and integrating one instantly obtains equality \eqref{2.9}.

In the general case we obtain this equality via closure. Note that by virtue of \eqref{1.5} if $u\in X^1(\Pi_T)$ then
\begin{equation}\label{2.10}
u\in C([0,T];L_p),\ u_x,u_y\in L_2(0,T;L_p)\quad \text{for any}\ p\in [2,+\infty)
\end{equation}
and this passage to the limit is easily justified.
\end{proof}

\begin{lemma}\label{L2.7}
Let $u_0\in H^2\cap H_0^1$, $f\in L_1(0,T;H^2\cap H_0^1)$. Then there exists a (unique) generalized solution to problem \eqref{2.1}, \eqref{1.2}, \eqref{1.3} $u\in X^2(\Pi_T)$. Moreover, for any $t\in (0,T]$
\begin{equation}\label{2.11}
\|u\|_{X^2(\Pi_t)}
\leq c(T,\delta) \left[\|u_0\|_{H^2}+\|f\|_{L_1(0,t_0;H^2)}\right]
\end{equation}
and
\begin{multline}\label{2.12}
\iint (u_{xx}^2+u^2_{xy}+u_{yy}^2)\,dxdy +
2\delta\int_0^t\!\! \iint (u_{xxx}^2+2u^2_{xxy}+2u^2_{xyy}+u^2_{yyy})\,dxdyd\tau \\=
\iint (u_{0xx}^2+u_{0xy}^2+u_{0yy}^2)\,dxdy  +
2\int_0^t\!\! \iint (f_{xx}u_{xx}+f_{xy}u_{xy}+f_{yy}u_{yy})\,dxdyd\tau.
\end{multline}
\end{lemma}

\begin{proof}
In the smooth case multiplying \eqref{2.1} by $2\bigl(u_{xxxx}(t,x,y)+u_{xxyy}(t,x,y)+u_{yyyy}(t,x,y)\bigr)$ and integrating over $\Sigma$ one obtains an equality
\begin{multline}\label{2.13}
\frac{d}{dt}\iint(u_{xx}^2+u^2_{xy}+u_{yy}^2) \,dxdy
+2\delta\iint(u^2_{xxx}+2u^2_{xxy}+2u^2_{xyy}+u^2_{yyy}) \,dxdy \\
=2\iint(f_{xx}u_{xx}+f_{xy}u_{xy}+f_{yy}u_{yy}) \,dxdy,
\end{multline}
whence \eqref{2.11} and \eqref{2.12} follows. Note also that the functions $u_0$ and $f$ can be approximated by corresponding functions satisfying the hypothesis of Lemma~\ref{L2.1}.
\end{proof}

\section{Global well-posedness}\label{S3}

\begin{definition}\label{D3.1}
Let $u_0\in L_2$.
A function $u\in L_\infty(0,T;L_2)\cap L_2(0,T;H^1)$ is called a weak solution to problem \eqref{1.1}--\eqref{1.3} in a layer $\Pi_T$ for some $T>0$ if for any function $\phi\in L_2(0,T;H^3\cap H_0^1)$, such that $\phi_t\in L_2(\Pi_T)$ and 
$\phi|_{t=T} =0$, the following equality holds:
\begin{multline}\label{3.1}
\iiint_{\Pi_T}\bigl[u(\phi_t+\phi_{xxx}+\phi_{xyy}) 
+\frac12 u^2\phi_{x}-\delta u_x\phi_x-\delta u_y\phi_y\bigr]\,dxdyd\tau \\
+\iint_\Sigma u_0\phi\big|_{t=0}\,dxdy=0.
\end{multline}
If $u$ is a weak solution to this problem in $\Pi_T$ for any $T>0$ it is called a weak solution to problem \eqref{1.1}--\eqref{1.3} in the layer $\Pi$.
\end{definition}

\begin{remark}\label{R3.2}
By virtue of \eqref{1.5} for any function $u\in L_\infty(0,T;L_2)\cap L_2(0,T;H^1)$
\begin{multline}\label{3.2}
\|u^2\|_{L_2(\Pi_T)} \leq
c\Bigl[\int_0^T\Bigl( \iint \bigl(|Du|^2+u^2\bigr)\,dxdy \iint u^2\, dxdy\Bigr) \,dt\Bigr]^{1/2} \\ \leq 
\|u\|_{L_2(0,T;H^1)}\|u\|_{L_\infty(0,T;L_2)} <\infty,
\end{multline}
therefore $u^2\phi_x\in L_1(\Pi_T)$.
\end{remark}

\begin{theorem}\label{T3.3}
Let $u_0\in L_2$. Then problem \eqref{1.1}--\eqref{1.3} has a unique weak solution $u$ in $\Pi$ such that $u\in X^0(\Pi_T)$ for any $T>0$. The mapping $u_0\mapsto u$ is Lipschitz continuous on any ball in the norm of the mapping from $L_2$ into $X^0(\Pi_T)$. Moreover, the function $\|u(t,\cdot,\cdot)\|^2_{L_2}$ is absolutely continuous for $t\geq 0$ and
\begin{equation}\label{3.3}
\frac d{dt} \iint u^2(t,x,y)\,dxdy +2\delta \iint \bigl(u_x^2(t,x,y)+u_y^2(t,x,y)\bigr)\,dxdy =0\quad \text{for a.e.}\ t>0.
\end{equation}
\end{theorem}

\begin{proof}
Consider first an auxiliary initial-boundary value problem in $\Pi$ with initial and boundary conditions \eqref{1.2}, \eqref{1.3} for an equation
\begin{equation}\label{3.4}
u_t+u_{xxx}+u_{xyy}-\delta(u_{xx}+u_{yy}) +(g_h(u))_x=0,
\end{equation}
where for $h\in (0,1]$
\begin{equation}\label{3.5}
g_h(u)\equiv \int_0^u \Bigl[\theta \eta(2-h|\theta|)+ \frac{2\sgn\theta}h \eta(h|\theta|-1)\Bigr]\,d\theta.
\end{equation}
Note that $g_h(u)=u^2/2$ if $|u|\leq 1/h$, $|g_h'(u)|\leq 2/h$ $\forall u\in\mathbb R$ and $|g'_h(u)|\leq 2|u|$ uniformly with respect to $h$.

We use the contraction principle to prove well-posedness of this problem in the space $X^0(\Pi_T)$ for any $T>0$.

Fix $T>0$. For $t_0\in(0,T]$ define a mapping $\Lambda$ on a set $X^0(\Pi_{t_0})$ as follows: $u=\Lambda v\in X^0(\Pi_{t_0})$ is a generalized solution to a linear problem
\begin{equation}\label{3.6}
u_t+u_{xxx}+u_{xyy}-\delta (u_{xx}+ u_{yy})=-(g_h(v))_x 
\end{equation}
in $\Pi_{t_0}$ with initial and boundary conditions \eqref{1.2}, \eqref{1.3}.

Note that $|g_h(v)|\leq 2|v|/h$ and, therefore, $g_h(u)\in L_2(\Pi_{t_0})$.
According to Lemma~\ref{L2.4} the mapping $\Lambda$ exists. Moreover, for functions $v,\widetilde v \in X^0(\Pi_{t_0})$
$$
\|g_h(v)-g_h(\widetilde v)\|_{L_2(\Pi_{t_0})}  \leq
\frac 2h\|v-\widetilde v\|_{L_2(\Pi_{t_0})} \leq
\frac {2t_0^{1/2}}h\|v-\widetilde v\|_{C([0,t_0];L_2)}.
$$
Inequality \eqref{2.3} yields that 
$$
\|\Lambda v - \Lambda\widetilde v\|_{X^0(\Pi_{t_0})} \leq 
\frac {c(T,\delta)}h t_0^{1/2} \|v-\widetilde v\|_{X^0(\Pi_{t_0})},
$$
that is for small $t_0$, depending only on $T$, $\delta$ and $h$, the mapping $\Lambda$ is the contraction in $X^0(\Pi_{t_0})$. Since $t_0$ is uniform with respect to $\|u_0\|_{L_2}$ by the standard argument we construct a solution to problem \eqref{3.4}, \eqref{1.2}, \eqref{1.3} $u_h\in X^0(\Pi_T)$.

Now establish appropriate estimates for functions $u_h$ uniform with respect to $h$. Equality \eqref{2.4} (where $f_0=f_2\equiv 0$, $f_1\equiv -g_h(u)$) provides that
\begin{multline}\label{3.7}
\iint u_h^2(t,x,y)\,dxdy +2\delta\int_0^t \!\! \iint (u_{hx}^2 +u_{hy}^2)\,dxdyd\tau = \iint u_0^2\,dxdy \\
+2\int_0^t \!\! \iint g_h(u_h)u_{hx}\,dxdyd\tau.
\end{multline}
Since the last integral is obviously equal to zero it follows from \eqref{3.7} that uniformly with respect to $h$
\begin{equation}\label{3.8}
\|u_h\|_{X^0(\Pi_T)} \leq c.
\end{equation}
Therefore, uniformly with respect to $h$
\begin{equation}\label{3.9}
\|g_h(u_h)\|_{L_2(\Pi_T)} \leq \|u_h^2\|_{L_2(\Pi_T)} \leq c.
\end{equation}
From estimates \eqref{3.8}, \eqref{3.9} and equation \eqref{3.4} itself follows that uniformly with respect to $h$
\begin{equation}\label{3.10}
\|u_{ht}\|_{L_2(0,T;H^{-2})} \leq c.
\end{equation}
Inequalities \eqref{3.2}, \eqref{3.8}--\eqref{3.10} by the standard argument provide existence of a weak solution $u$ to 
problem \eqref{1.1}--\eqref{1.3} in $L_\infty(0,T;L_2)\cap L_2(0,T;H^1_0)$.

Next, Lemma~\ref{L2.3} (where $f_0=f_2\equiv 0$, $f_1\equiv -u^2/2\in L_2(\Pi_T)$) provides that (after possible change on a set of the zero measure) $u\in C([0,T];L_2)$ and similarly to \eqref{3.7}
\begin{equation}\label{3.11}
\iint u^2(t,x,y)\,dxdy +2\delta \int_0^t \!\!\iint |Du|^2\,dxdyd\tau = \iint u_0^2\,dxdy. 
\end{equation}
In particular, equality \eqref{3.11} yields that the function $\|u(t,\cdot,\cdot)\|^2_{L_2}$ is absolutely continuous and equality \eqref{3.3} is satisfied.

Finally, establish properties of uniqueness and continuous dependence. Let $u$ and $\widetilde u$ be two solutions in the considered space corresponding to initial data $u_0$ and $\widetilde u_0$, $v\equiv u-\widetilde u$, 
$v_0\equiv u_0 -\widetilde u_0$. Then the function $v$ is a weak solution to a linear problem
\begin{gather}\label{3.12}
v_t+v_{xxx}+v_{xyy}-\delta(v_{xx}+v_{yy}) =\frac 12 (\widetilde u^2 -u^2)_x,\\
\label{3.13}
v\big|_{t=0}=v_0,\qquad v\big|_{y=0}=v\big|_{y=L}=0.
\end{gather}
Obviously the hypothesis of Lemma~\ref{L2.3} are satisfied for this problem and equality \eqref{2.4} provides that
\begin{multline*}
\iint v^2(t,x,y)\,dxdy +2\delta \int_0^t\!\! \iint |Dv|^2\,dxdyd\tau = 
\iint v_0^2\,dxdy \\+
\int_0^t\!\! \iint (u+\widetilde u)vv_x\,dxdyd\tau.
\end{multline*}
Here
\begin{multline*}
\iint |uvv_x|\,dxdy \leq
\Bigl(\iint u^4\,dxdy\Bigr)^{1/4} \Bigl(\iint v^4\,dxdy\Bigr)^{1/4} 
\Bigl(\iint v_x^2\,dxdy\Bigr)^{1/2}\\
\leq c\Bigl(\iint (|Du|^2+u^2)\,dxdy\Bigr)^{1/4} \Bigl(\iint u^2\,dxdx\Bigr)^{1/4} \\ \times
\Bigl(\iint (|Dv|^2+v^2)\,dxdy\Bigr)^{3/4} \Bigl(\iint v^2\,dxdy\Bigr)^{1/4} \\
\leq \varepsilon \iint |Dv|^2\,dxdy +c(\varepsilon)\iint (|Du|^2+u^2)\,dxdy
\iint v^2\,dxdy,
\end{multline*}
where $\varepsilon>0$ can be chosen arbitrarily small. With use of \eqref{3.11} we finish the proof of the theorem.
\end{proof}

\begin{theorem}\label{T3.4}
Let $u_0\in H^1_0$. Then problem \eqref{1.1}--\eqref{1.3} has a unique weak solution $u$ in $\Pi$ such that $u\in X^1(\Pi_T)$ for any $T>0$. The mapping $u_0\mapsto u$ in Lipschitz continuous on any ball in the norm of the mapping from $H^1$ into $X^1(\Pi_T)$ and 
\begin{equation}\label{3.14}
\|u\|_{X^1(\Pi_T)} \leq \varkappa_1(T,\|u_0\|_{L_2}) \|u_0\|_{H^1},
\end{equation}
where the positive function $\varkappa_1$ is nondecreasing with respect to its arguments. Moreover, the function 
$\bigl\||Du|(t,\cdot,\cdot)\bigr\|^2_{L_2}$ is absolutely continuous for $t\geq 0$ and
\begin{multline}\label{3.15}
\frac d{dt} \iint (u_x^2+u_y^2)\,dxdy 
+2\delta \iint (u^2_{xx}+2u^2_{xy}+u^2_{yy})\,dxdy  \\
= 2\iint uu_x(u_{xx}+u_{yy})\,dxdy \quad \text{for a.e.}\ t>0.
\end{multline}
\end{theorem}

\begin{proof}
As in the proof of Theorem~\ref{T3.3} we first apply the contraction principle but for the original problem. To this end consider an initial-boundary value problem for a linear equation
\begin{equation}\label{3.16}
u_t+u_{xxx}+u_{xyy}-\delta(u_{xx}+u_{yy}) = -vv_x
\end{equation}
with initial and boundary conditions \eqref{1.2}, \eqref{1.3}. Again fix $T>0$. For $t_0\in (0,T]$ let $v\in X^1(\Pi_{t_0})$ 
and $u=\Lambda v$ be a solution to this problem from the space $X^1(\Pi_{t_0})$ also.
Note that by virtue of \eqref{1.5}
\begin{multline}\label{3.17}
\|vv_x\|_{L_2(\Pi_{t_0})} \leq
\Bigl[\int_0^{t_0} \Bigl(\sup\limits_{(x,y)\in\Sigma} v^2 \iint v_x^2\,dxdy\Bigr)\,dt\Bigr]^{1/2} \\ \leq
c\Bigl[\int_0^{t_0} \Bigl(\iint \bigl(|D^2v|^2+v^2\bigr)\,dxdy \iint v^2\,dxdy\Bigr)^{1/2}\,dt\Bigr]^{1/2}
\sup\limits_{t\in(0,t_0)} \Bigl(\iint v_x^2\,dxdy\Bigr)^{1/2} \\ \leq
ct_0^{1/4} \|v\|_{L_2(0,t_0;H^2)}^{1/2} \|v\|_{C([0,t_0];H^1)}^{3/2} \leq ct_0^{1/4}\|v\|^2_{X^1(\Pi_{t_0})}
\end{multline}
and similarly
\begin{equation}\label{3.18}
\|vv_x-\widetilde v\widetilde v_x\|_{L_2(\Pi_{t_0})} \leq 
ct_0^{1/4}\bigl(\|v\|_{X^1(\Pi_{t_0})}+\|\widetilde v\|_{X^1(\Pi_{t_0})}\bigr) 
\|v-\widetilde v\|_{X^1(\Pi_{t_0})}.
\end{equation}
In particular, the hypothesis of Lemma~\ref{L2.5} is satisfied (for $f_0\equiv 0$, $f_1\equiv -vv_x$) and, therefore, the mapping $\Lambda$ exists. Moreover, inequalities \eqref{2.6}, \eqref{3.17}, \eqref{3.18} provide that
\begin{equation}\label{3.19}
\|\Lambda v\|_{X^1(\Pi_{t_0})} \leq c(T,\delta)\bigl[\|u_0\|_{H^1} + t_0^{1/4}\|v\|^2_{X^1(\Pi_{t_0})}\bigr],
\end{equation}
\begin{multline}\label{3.20}
\|\Lambda v -\lambda \widetilde v\|_{X^1(\Pi_{t_0})} \leq c(T,\delta)\bigl[\|u_0-\widetilde u_0\|_{H^1}  \\+ 
t_0^{1/4}\bigl(\|v\|_{X^1(\Pi_{t_0})}+\|\widetilde v\|_{X^1(\Pi_{t_0})}\bigr) 
\|v-\widetilde v\|_{X^1(\Pi_{t_0})}\bigr].
\end{multline}
Local well-posedness of problem \eqref{1.1}--\eqref{1.3} on the time interval $(0,t_0)$ depending on $\|u_0\|_{H^1}$ follows from \eqref{3.19}, \eqref{3.20} by the standard argument.

In order to extend this local solution to an arbitrary time interval establish the corresponding a priori estimate. Let 
$u\in X^1(\Pi_{T'})$ be a solution to problem \eqref{1.1}--\eqref{1.3}. Again apply Lemma~\ref{L2.5}, where 
$f_0\equiv 0$, $f_1\equiv -uu_x$. It follows from equality \eqref{2.7} that
\begin{multline}\label{3.21}
\iint (u_x^2+u_y^2)\,dxdy +2\delta\int_0^t\!\! \iint (u_{xx}^2+2u^2_{xy}+u^2_{yy})\,dxdyd\tau =
\iint (u_{0\,x}^2+u_{0\,y}^2)\,dxdy  \\+
2\int_0^t\!\! \iint uu_x(u_{xx}+u_{yy})\,dxdyd\tau.
\end{multline}
Next, apply Lemma~\ref{L2.6}, then equality \eqref{2.9} yields that
\begin{multline}\label{3.22}
-\frac 13 \iint u^3(t,x,y)\,dxdy
+2\int_0^t\!\! \iint uu_x(u_{xx}+u_{yy})\,dxdyd\tau \\
+\delta \int_0^t\!\! \iint u^2(u_{xx}+u_{yy})\,dxdyd\tau =
-\frac 13 \iint u_0^3\,dxdy + \int_0^t\!\! \iint u^3u_x\,dxdyd\tau.
\end{multline}
Summing \eqref{3.21} and \eqref{3.22} provides an equality
\begin{multline}\label{3.23}
\iint \Bigl(u_x^2+u_y^2-\frac 13 u^3 \Bigr)\,dxdy
+2\delta\int_0^t\!\! \iint (u_{xx}^2+2u^2_{xy}+u^2_{yy})\,dxdyd\tau \\ +
\delta \int_0^t\!\! \iint u^2(u_{xx}+u_{yy})\,dxdyd\tau =
\iint \Bigl(u_{0\,x}^2+u_{0\,y}^2-\frac 13 u_0^3\Bigr)\,dxdy.
\end{multline}
By virtue of \eqref{1.5} and \eqref{3.11}
\begin{multline}\label{3.24}
\iint |u|^3\,dxdy \leq c\Bigl(\iint (|Du|^2+u^2)\,dxdy\Bigr)^{1/2} \iint u^2\,dxdy \\ \leq
\varepsilon \iint |Du|^2\,dxdy +c(\varepsilon)\left(\|u_0\|^2_{L_2}+\|u_0\|^4_{L_2}\right),
\end{multline}
\begin{multline}\label{3.25}
\Bigl| \iint u^2(u_{xx}+u_{yy})\,dxdy \Bigr| \leq 
c\Bigl(\iint |D^2u|^2\,dxdx\Bigr)^{1/2} \Bigl(\iint u^4\,dxdy\,\Bigr)^{1/2} \\ \leq
c_1\Bigl(\iint (|D^2u|^2+u^2)\,dxdx\Bigr)^{3/4} \Bigl(\iint u^2\,dxdx\Bigr)^{3/4} \\ \leq
\varepsilon \iint |D^2u|^2\,dxdy +c(\varepsilon) \left(\|u_0\|^3_{L_2}+\|u_0\|^6_{L_2}\right),
\end{multline}
where $\varepsilon>0$ can be chosen arbitrarily small. Combining \eqref{3.23}--\eqref{3.25} yields an inequality
\begin{equation}\label{3.26}
\sup\limits_{t\in (0,T')} \iint |Du|^2\,dxdy + \int_0^{T'}\!\! \iint |D^2u|^2\,dxdydt  \leq 
c(T')\left(1+\|u_0\|_{L_2}^4\right)\|u_0\|_{H^1}^2.
\end{equation}
This estimate provides the desired global well-posedness and, moreover, estimate \eqref{3.14}.

Finally, note that for the solution $u\in X^1(\Pi_T)$ similarly to \eqref{2.10} $uu_x\in L_2(\Pi_T)$ and, therefore,  
$\iint uu_x(u_{xx}+u_{yy})\,dxdy \in L_1(0,T)$. As a result, it follows from \eqref{3.21} that 
$\bigl\||Du|(t,\cdot,\cdot)\bigr\|^2_{L_2}$ is absolutely continuous and equality \eqref{3.15} holds.
\end{proof}

\begin{corollary}\label{C3.5}
Let $u_0\in H^s$ for certain $s\in (0,1)$ and, in addition, $u_0\big|_{y=0}=u_0\big|_{y=L}= 0$ if $s>1/2$ and 
$(y^{-1/2}+(L-y)^{-1/2})u_0\in L_2$ if $s=1/2$. Then problem \eqref{1.1}--\eqref{1.3} has a unique weak solution $u$ in $\Pi$ such that $u \in X^s(\Pi_T)$ for any $T>0$. Moreover,
\begin{equation}\label{3.27}
\|u\|_{X^s(\Pi_T)} \leq \varkappa_s(T,\|u_0\|_{L_2}) \|u_0\|_{H^s},
\end{equation}
where the positive function $\varkappa_s$ is nondecreasing with respect to its arguments.
\end{corollary}

\begin{proof}
Results from \cite{LM} ensure that under the hypothesis of the corollary the function $u_0$ belong to spaces which form the real interpolation scale $(\cdot,\cdot)_{\theta,2}$. The spaces $X^s(\Pi_T)$ also form the same interpolation scale. Then the corollary succeeds from Lemma~\ref{L1.3} and Theorems~\ref{T3.3} and~\ref{T3.4}. 
\end{proof}

\begin{theorem}\label{T3.6}
Let $u_0\in H^2\cap H^1_0$. Then problem \eqref{1.1}--\eqref{1.3} has a unique weak solution $u$ in $\Pi$ such that $u\in X^2(\Pi_T)$ for any $T>0$. The mapping $u_0\mapsto u$ is Lipschitz continuous on any ball in the norm of the mapping from $H^2$ into $X^2(\Pi_T)$ and 
\begin{equation}\label{3.28}
\|u\|_{X^2(\Pi_T)} \leq \varkappa_2(T,\|u_0\|_{H^1}) \|u_0\|_{H^2},
\end{equation}
where the positive function $\varkappa_2$ is nondecreasing with respect to its arguments. Moreover, the function 
$\bigl\||D^2u|(t,\cdot,\cdot)\bigr\|^2_{L_2}$ is absolutely continuous for $t\geq 0$ and
\begin{multline}\label{3.29}
\frac d{dt} \iint (u_{xx}^2+u_{xy}^2+u_y^2)\,dxdy 
+2\delta \iint (u^2_{xxx}+2u^2_{xxy}+2u^2_{xyy}+u^2_{yyy})\,dxdy  \\
= -2\iint \bigl((uu_x)_{xx}u_{xx}+(uu_x)_{xy}u_{xy}+(uu_x)_{yy}u_{yy}\bigr)\,dxdy \quad \text{for a.e.}\ t>0.
\end{multline}
\end{theorem}

\begin{proof}
As in the proof of Theorem~\ref{T3.4} consider linear initial-boundary value problem \eqref{3.16}, \eqref{1.2}, \eqref{1.3}. For $T>0$ and $t_0\in (0,T]$ let $v\in X^2(\Pi_{t_0})$ 
and $u=\Lambda v$ be a solution to this problem from the space $X^2(\Pi_{t_0})$. In order to apply Lemma~\ref{L2.7} we have to estimate $f=-vv_x$ in $L_1(0,t_0;H^2)$. For example, by virtue of \eqref{1.5}
\begin{multline}\label{3.30}
\|vv_{xxx}\|_{L_1(0,t_0;L_2)} \leq
\int_0^{t_0} \sup\limits_{(x,y)\in\Sigma} |v| \Bigl(\iint v_{xxx}^2\,dxdy\Bigr)^{1/2}\,dt \\ \leq
c\int_0^{t_0} \Bigl(\iint \bigl(|D^2v|^2+v^2\bigr)\,dxdy\Bigr)^{1/2} \Bigl(\iint v_{xxx}^2\,dxdy\Bigr)^{1/2}\,dt \\ \leq
ct_0^{1/2} \|v\|_{C([0,t_0];H^2)} \|v\|_{L_2(0,t_0;H^3)} \leq ct_0^{1/2}\|v\|^2_{X^2(\Pi_{t_0})},
\end{multline}
\begin{multline}\label{3.31}
\|v_xv_{xx}\|_{L_1(0,t_0;L_2)} \leq
\int_0^{t_0} \sup\limits_{(x,y)\in\Sigma} |v_x| \Bigl(\iint v_{xx}^2\,dxdy\Bigr)^{1/2}\,dt \\ \leq
c\int_0^{t_0} \Bigl(\iint \bigl(|D^3v|^2+v_x^2\bigr)\,dxdy\Bigr)^{1/2} \Bigl(\iint v_{xx}^2\,dxdy\Bigr)^{1/2}\,dt \\ \leq
ct_0^{1/2} \|v\|_{L_2(0,t_0;H^3)} \|v\|_{C([0,t_0];H^2)} \leq ct_0^{1/2}\|v\|^2_{X^2(\Pi_{t_0})}.
\end{multline}
Other terms can be estimated in a similar way and, 
therefore, the hypothesis of Lemma~\ref{L2.7} is satisfied and the mapping $\Lambda$ exists. Moreover, inequalities \eqref{2.11}, \eqref{3.30}, \eqref{3.31} provide that
\begin{equation}\label{3.32}
\|\Lambda v\|_{X^2(\Pi_{t_0})} \leq c(T,\delta)\bigl[\|u_0\|_{H^2} + t_0^{1/2}\|v\|^2_{X^2(\Pi_{t_0})}\bigr].
\end{equation}
Moreover, one can similarly show that
\begin{multline}\label{3.33}
\|\Lambda v -\lambda \widetilde v\|_{X^2(\Pi_{t_0})} \leq c(T,\delta)\bigl[\|u_0-\widetilde u_0\|_{H^2}  \\+ 
t_0^{1/2}\bigl(\|v\|_{X^2(\Pi_{t_0})}+\|\widetilde v\|_{X^2(\Pi_{t_0})}\bigr) 
\|v-\widetilde v\|_{X^2(\Pi_{t_0})}\bigr].
\end{multline}
Local well-posedness of problem \eqref{1.1}--\eqref{1.3} on the time interval $(0,t_0)$ depending on $\|u_0\|_{H^2}$ follows from \eqref{3.32}, \eqref{3.33} by the standard argument.

In order to extend this local solution to an arbitrary time interval establish the corresponding a priori estimate. Let 
$u\in X^2(\Pi_{T'})$ be a solution to problem \eqref{1.1}--\eqref{1.3}. Again apply Lemma~\ref{L2.7}, where 
$f\equiv -uu_x$. It follows from equality \eqref{2.12} that
\begin{multline}\label{3.34}
\iint (u_{xx}^2+u_{xy}^2+u_y^2)\,dxdy 
+2\delta\int_0^{t_0} \iint (u^2_{xxx}+2u^2_{xxy}+2u^2_{xyy}+u^2_{yyy})\,dxdyd\tau  \\
=\iint (u_{0\,xx}^2+u_{0\,xy}^2+u_{0\,yy}^2)\,dxdy  \\ -2 \int_0^{t_0}\iint \bigl((uu_x)_{xx}u_{xx}+(uu_x)_{xy}u_{xy}+(uu_x)_{yy}u_{yy}\bigr)\,dxdyd\tau.
\end{multline}
Here by virtue of \eqref{1.5} and \eqref{3.14}
\begin{multline}\label{3.35}
\iint |uu_{xx}u_{xxx}|\,dxdy  \leq 
\Bigl(\iint u^4\,dxdy \iint u_{xx}^4\,dxdy\Bigr)^{1/4} \Bigl(\iint u_{xxx}^2\,dxdy\Bigr)^{1/2} \\ \leq
c_1\Bigl(\iint (|Du|^2+u^2)\,dxdy\Bigr)^{1/2} \Bigl(\iint u_{xx}^2\,dxdy\,\Bigr)^{1/4}  
\Bigl(\iint \bigl(|D^3u|^2+|D^2u|^2\bigr)\,dxdy\Bigr)^{3/4}\\ \leq
\varepsilon \iint |D^3u|^2\,dxdy +c(\varepsilon) \Bigl(\iint (|Du|^2+u^2)\,dxdy\Bigr)^2 \iint |D^2u|^2\,dxdy \\ \leq
\varepsilon \iint |D^3u|^2\,dxdy +c(\varepsilon,\|u_0\|_{H^1}) \iint |D^2u|^2\,dxdy,
\end{multline}
\begin{multline}\label{3.36}
\iint |u_x|u_{xx}^2\,dxdy \leq 
\Bigl(\iint u_x^2\,dxdy\Bigr)^{1/2} \Bigl(\iint u_{xx}^4\,dxdy\,\Bigr)^{1/2} \\ \leq
c\Bigl(\iint u_x^2\,dxdy\Bigr)^{1/2} \Bigl(\iint u_{xx}^2\,dxdy\,\Bigr)^{1/2}  
\Bigl(\iint \bigl(|D^3u|^2+|D^2u|^2\bigr)\,dxdy\Bigr)^{1/2}\\ \leq
\varepsilon \iint |D^3u|^2\,dxdy +c(\varepsilon) \iint u_x^2\,dxdy \iint |D^2u|^2\,dxdy \\ \leq
\varepsilon \iint |D^3u|^2\,dxdy +c(\varepsilon,\|u_0\|_{H^1}) \iint |D^2u|^2\,dxdy,
\end{multline}
where $\varepsilon>0$ can be chosen arbitrarily small. Other terms in the right side of \eqref{3.34} are estimated in a similar way. Combining \eqref{3.34}--\eqref{3.36} yields an inequality
\begin{equation}\label{3.37}
\sup\limits_{t\in (0,T')} \iint |D^2u|^2\,dxdy + \int_0^{T'}\!\! \iint |D^3u|^2\,dxdydt  \leq 
c(T',\|u_0\|_{H^1})\|u_0\|_{H^2}^2.
\end{equation}
This estimate provides the desired global well-posedness and, moreover, estimate \eqref{3.28}.

Finally, note that for the solution $u\in X^2(\Pi_T)$ estimates \eqref{3.35}, \eqref{3.36} ensure that 
$\iint \bigl((uu_x)_{xx}u_{xx}+(uu_x)_{xy}u_{xy}+(uu_x)_{yy}u_{yy}\bigr)\,dxdy \in L_1(0,T)$. Therefore, it follows from \eqref{3.34} that 
$\bigl\||D^2u|(t,\cdot,\cdot)\bigr\|^2_{L_2}$ is absolutely continuous and equality \eqref{3.29} holds.
\end{proof}

\begin{corollary}\label{C3.7}
Let $u_0\in H^s\cap H_0^1$ for a certain $s\in (1,2)$. Then problem \eqref{1.1}--\eqref{1.3} has a unique weak solution $u$ in $\Pi$ such that $u \in X^s(\Pi_T)$ for any $T>0$. Moreover,
\begin{equation}\label{3.38}
\|u\|_{X^s(\Pi_T)} \leq \varkappa_s(T,\|u_0\|_{H^1}) \|u_0\|_{H^s},
\end{equation}
where the positive function $\varkappa_s$ is nondecreasing with respect to its arguments.
\end{corollary}

\begin{proof}
The hypothesis of the corollary provides that the consideed spaces for $u_0$ form the real interpolation scale $(\cdot,\cdot)_{\theta,2}$. The spaces $X^s(\Pi_T)$ also form the same interpolation scale. Then the corollary succeeds from Lemma~\ref{L1.3} and Theorems~\ref{T3.4} and~\ref{T3.6}. 
\end{proof}

\begin{corollary}\label{C3.8}
Let the hypothesis of Theorem~\ref{T3.3} be satisfied. Consider the unique weak solution to problem \eqref{1.1}--\eqref{1.3} $u$ in $\Pi$ such that $u\in X^0(\Pi_T)$ for any $T>0$. Then for any $T>0$ and $t_0\in (0,T)$
the function $u\in X^2(\Pi_{t_0,T})$.
\end{corollary}

\begin{proof}
Since $u\in L_2(0,T;H^1)$ for any $t_0\in (0,T)$ there exists $t_1\in (0,t_0)$ such that $u(t_1,\cdot,\cdot)\in H^1_0$.
Consider the function $u$ as a weak solution to an initial-boundary value problem in $\Pi_{t_1,T}$ for equation \eqref{1.1} with initial data $u_0=u(t_1,\cdot,\cdot)$ and boundary condition \eqref{1.3}. The hypothesis of Theorem~\ref{T3.4} are satisfied for this problem, therefore, $u\in X^1(\Pi_{t_1,T})$.

Similarly there exists $t_2\in (t_1,t_0)$ such that $u(t_2,\cdot,\cdot)\in H^2\cap H^1_0$. Now consider $u$ as a weak solution to a similar initial-boundary value problem but in $\Pi_{t_2,T}$, then according to Theorem~\ref{T3.6} $u\in X^2(\Pi_{t_2,T})$.
\end{proof}

\section{Long-time decay}\label{S4}

\begin{lemma}\label{L4.1}
Let $u_0\in L_2$. Then a weak solution to problem \eqref{1.3}--\eqref{1.3} from the space $X^0(\Pi_T)$ for any $T>0$ satisfies inequality 
\begin{equation}\label{4.1}
\|u(t,\cdot,\cdot)\|_{L_2} \leq e^{-\delta\pi^2 L^{-2}t}\|u_0\|_{L_2}\qquad \forall\ t\geq 0.
\end{equation}
\end{lemma}

\begin{proof}
Consider equality \eqref{3.3}. With the use of inequality \eqref{1.6} we derive that
\begin{equation}\label{4.2}
\iint u_y^2\,dxdy \geq  \frac{\pi^2}{L^2}\iint u^2\,dxdy 
\end{equation}
and it follows \eqref{3.3} that
\begin{equation}\label{4.3}
\frac d{dt} \iint u^2\,dxdy +\frac {2\delta\pi^2}{L^2} \iint u^2\,dxdy \leq 0,
\end{equation}
which yields \eqref{4.1}.
\end{proof}

\begin{lemma}\label{L4.2}
Let $u_0\in H^1_0$. Then a weak solution to problem \eqref{1.3}--\eqref{1.3} from the space $X^1(\Pi_T)$ for any $T>0$ satisfies inequality \eqref{1.4} for $s=1$.
\end{lemma}

\begin{proof}
Consider equality \eqref{3.15}. By virtue of \eqref{1.5}
\begin{multline}\label{4.4}
\Bigl| \iint uu_x(u_{xx}+u_{yy})\,dxdy \Bigr|  \leq
c\sup\limits_{(x,y)\in\Sigma}|u|\Bigl(\iint u_x^2\,dxdy\Bigr)^{1/2}\Big(\iint|D^2u|^2\,dxdy\Bigr)^{1/2} \\ \leq
c_1\iint \bigl(|D^2u|^2+u^2\bigr)\,dxdy \iint u^2\,dxdy.
\end{multline}
Choose $T_1=T_1(\|u_0\|_{L_2})>0$ such that according to \eqref{4.1}
\begin{equation}\label{4.5}
\iint u^2(t,x,y)\,dxdy \leq \min\bigl(\frac{\delta}{2c_1},\frac{\delta\pi^2}{2c_1L^2}\bigr)\qquad \forall\ t\geq T_1,
\end{equation}
where $c_1$ is the constant from \eqref{4.4}. Then summing \eqref{3.15}, \eqref{3.3} and applying \eqref{1.6} yields
\begin{equation}\label{4.6}
\frac{d}{dt} \iint \bigl(|Du|^2+u^2)\,dxdy + \delta\iint |Du|^2\,dxdy +\frac{\delta\pi^2}{2L^2} \iint u^2\,dxdy \leq 0
\quad \forall t\geq T_1.
\end{equation}
Since according to \eqref{3.14}
\begin{equation}\label{4.7}
\|u(T_1,\cdot,\cdot)\|_{H^1} \leq c(T_1,\|u_0\|_{L_2})\|u_0\|_{H^1}
\end{equation}
inequality \eqref{4.6} provides the desired result.
\end{proof}

\begin{lemma}\label{L4.3}
Let $u_0\in H^2\cap H_0^1$. Then a weak solution to problem \eqref{1.3}--\eqref{1.3} from the space $X^2(\Pi_T)$ for any $T>0$ satisfies inequality \eqref{1.4} for $s=2$.
\end{lemma}

\begin{proof}
Consider equality \eqref{3.29}. Similarly to \eqref{3.35}, \eqref{3.36}
\begin{multline*}
\iint |uu_{xx}u_{xxx}|\,dxdy + \iint |u_x|u_{xx}^2\,dxdy   \\ \leq
\varepsilon \iint |D^3u|^2\,dxdy +c(\varepsilon) \iint (|Du|^2+u^2)\,dxdy \iint |D^2u|^2\,dxdy,
\end{multline*}
where $\varepsilon>0$ can be chosen arbitrarily small. Other terms in the right side of \eqref{3.29} are estimated in the same way. Therefore, it follows from this equality that
\begin{equation}\label{4.8}
\frac{d}{dt} \iint |D^2u|^2\,dxdy \leq c_2 \iint (|Du|^2+u^2)\,dxdy \iint |D^2u|^2\,dxdy.
\end{equation}
Summing this inequality with \eqref{3.3}, \eqref{3.15} and applying \eqref{1.6}, \eqref{4.4} yields that
\begin{multline}\label{4.9}
\frac{d}{dt} \iint \bigl(|D^2u|^2+|Du|^2+u^2\bigr)\,dxdy +\delta\iint\bigl(2|D^2u|^2+|Du|^2\bigr)\,dxdy
+\frac{\delta\pi^2}{L^2} \iint u^2\,dxdy  \\ \leq
c_1\iint \bigl(|D^2u|^2+u^2\bigr)\,dxdy \iint u^2\,dxdy \\+
c_2 \iint (|Du|^2+u^2)\,dxdy \iint |D^2u|^2\,dxdy.
\end{multline}
Choose $T_2=T_2(\|u_0\|_{H^1})$ such that according to \eqref{1.4} for $s=1$
$$
\iint \bigl(|Du|^2+u^2\bigr)\,dxdy \leq 
\min\bigl(\frac{\delta}{2c_1},\frac{\delta\pi^2}{2c_1L^2},\frac{\delta}{2c_2}\bigr)\qquad \forall\ t\geq T_2,
$$
then it follows from \eqref{4.9} that
\begin{multline}\label{4.10}
\frac{d}{dt} \iint \bigl(|D^2u|^2+|Du|^2+u^2\bigr)\,dxdy \\+
\delta\iint \bigl(|D^2u|^2+|Du|^2\bigr)\,dxdy +\frac{\delta\pi^2}{2L^2} \iint u^2\,dxdy \leq 0
\quad \forall t\geq T_2.
\end{multline}
Since according to \eqref{3.28}
\begin{equation}\label{4.7}
\|u(T_2,\cdot,\cdot)\|_{H^2} \leq c(T_2,\|u_0\|_{H^1})\|u_0\|_{H^2}
\end{equation}
inequality \eqref{4.10} provides the desired result.
\end{proof}

Now we can finish the proof of Theorem~\ref{T1.1}.

\begin{proof}[Proof of Theorem~\ref{T1.1}]
It still remains to prove exponential decay of solutions for $s\in (0,1)$ and $s\in (1,2)$.

Let $s\in (0,1)$. By virtue of Corollary~\ref{C3.5} it is sufficient to prove \eqref{1.4} for $t\geq 1$. Then Corollary~\ref{C3.8} yields that $u\in X^1(\Pi_{1,T})$ for any $T>1$. We have with use of \eqref{1.4} for $s=0$, $s=1$ and \eqref{3.14} that
\begin{multline}\label{4.12}
\|u(t,\cdot,\cdot)\|_{H^s} \leq c(s)\|u\|_{L_2}^{1-s}\|u\|_{H^1}^s \\ \leq 
c_1e^{-\beta(0)(1-s)t}\|u_0\|_{L_2}^{1-s} e^{-\beta(1)s(t-1)}\|u(1,\cdot,\cdot)\|_{H^1}^s \\ \leq
c_1e^{\beta(1)s}e^{-(\beta(0)(1-s)+\beta(1)s)t}\varkappa_1^s(1,\|u_0\|_{L_2})\|u_0\|_{L_2},
\end{multline}
whence \eqref{1.4} follows. The case $s\in (1,2)$ is considered in a similar way with use of \eqref{1.4} for $s=1$, $s=2$ and \eqref{3.28} .
\end{proof}

\bibliographystyle{amsplain}

\end{document}